\documentclass[12pt]{amsart}
\usepackage{amssymb,latexsym,tikz}
\usetikzlibrary{arrows,decorations.markings}
\tikzstyle arrowstyle=[scale=1]
\tikzstyle directed=[postaction={decorate,
decoration={markings,mark=at position .65 with {\arrow[arrowstyle]{stealth}}}}]

\usepackage{extdash}
\usepackage{mathtools} 
\usepackage{amscd}
\usepackage{xypic}
\usepackage[utf8]{inputenc}
\usepackage[top=40mm, bottom=35mm, left=40mm, right=35mm]{geometry}

\begin{document}

\title [The five-sequence of adjoints for simplicial complexes]
{The five-sequence of adjoints for combinatorial  simplicial complexes}

\author{Gunnar Fl{\o}ystad}
\address{Matematisk Institutt\\
         Postboks 7803\\
         5020 Bergen}
\email{gunnar@mi.uib.no}


\date{\today}


\theoremstyle{plain}
\newtheorem{theorem}{Theorem}[section]
\newtheorem{corollary}[theorem]{Corollary}
\newtheorem*{main}{Main Theorem}
\newtheorem{lemma}[theorem]{Lemma}
\newtheorem{proposition}[theorem]{Proposition}
\newtheorem{conjecture}[theorem]{Conjecture}
\newtheorem{theoremp}{Theorem}

\theoremstyle{definition}
\newtheorem{definition}[theorem]{Definition}
\newtheorem{fact}[theorem]{Fact}
\newtheorem{obs}[theorem]{Observation}
\newtheorem{definisjon}[theorem]{Definisjon}
\newtheorem{problem}[theorem]{Problem}
\newtheorem{condition}[theorem]{Condition}

\theoremstyle{remark}
\newtheorem{notation}[theorem]{Notation}
\newtheorem{remark}[theorem]{Remark}
\newtheorem{example}[theorem]{Example}
\newtheorem{claim}{Claim}
\newtheorem{observation}[theorem]{Observation}
\newtheorem{question}[theorem]{Question}


\newcommand{\psp}[1]{{{\bf P}^{#1}}}
\newcommand{\psr}[1]{{\bf P}(#1)}
\newcommand{\op}{{\mathcal O}}
\newcommand{\opw}{\op_{\psr{W}}}

\newcommand{\ini}[1]{\text{in}(#1)}
\newcommand{\gin}[1]{\text{gin}(#1)}
\newcommand{\kr}{{\Bbbk}}
\newcommand{\pd}{\partial}
\newcommand{\vardel}{\partial}
\renewcommand{\tt}{{\bf t}}


\newcommand{\coh}{{{\text{{\rm coh}}}}}


\newcommand{\modv}[1]{{#1}\text{-{mod}}}
\newcommand{\modstab}[1]{{#1}-\underline{\text{mod}}}

\newcommand{\sut}{{}^{\tau}}
\newcommand{\sumit}{{}^{-\tau}}
\newcommand{\til}{\thicksim}

\newcommand{\totp}{\text{Tot}^{\prod}}
\newcommand{\dsum}{\bigoplus}
\newcommand{\dprod}{\prod}
\newcommand{\lsum}{\oplus}
\newcommand{\lprod}{\Pi}

\newcommand{\La}{{\Lambda}}

\newcommand{\sirstj}{\circledast}

\newcommand{\she}{\EuScript{S}\text{h}}
\newcommand{\cm}{\EuScript{CM}}
\newcommand{\cmd}{\EuScript{CM}^\dagger}
\newcommand{\cmri}{\EuScript{CM}^\circ}
\newcommand{\cler}{\EuScript{CL}}
\newcommand{\clerd}{\EuScript{CL}^\dagger}
\newcommand{\clerri}{\EuScript{CL}^\circ}
\newcommand{\gor}{\EuScript{G}}
\newcommand{\cF}{\mathcal{F}}
\newcommand{\cG}{\mathcal{G}}
\newcommand{\cM}{\mathcal{M}}
\newcommand{\cE}{\mathcal{E}}
\newcommand{\cI}{\mathcal{I}}
\newcommand{\cP}{\mathcal{P}}
\newcommand{\cK}{\mathcal{K}}
\newcommand{\cS}{\mathcal{S}}
\newcommand{\cC}{\mathcal{C}}
\newcommand{\cO}{\mathcal{O}}
\newcommand{\cJ}{\mathcal{J}}
\newcommand{\cU}{\mathcal{U}}
\newcommand{\cQ}{\mathcal{Q}}
\newcommand{\cX}{\mathcal{X}}
\newcommand{\cY}{\mathcal{Y}}
\newcommand{\cZ}{\mathcal{Z}}
\newcommand{\cV}{\mathcal{V}}

\newcommand{\mm}{\mathfrak{m}}

\newcommand{\dlim} {\varinjlim}
\newcommand{\ilim} {\varprojlim}

\newcommand{\CM}{\text{CM}}
\newcommand{\Mon}{\text{Mon}}


\newcommand{\Kom}{\text{Kom}}


\newcommand{\EH}{{\mathbf H}}
\newcommand{\res}{\text{res}}
\newcommand{\Hom}{\text{Hom}}
\newcommand{\inhom}{{\underline{\text{Hom}}}}
\newcommand{\Ext}{\text{Ext}}
\newcommand{\Tor}{\text{Tor}}
\newcommand{\ghom}{\mathcal{H}om}
\newcommand{\gext}{\mathcal{E}xt}
\newcommand{\id}{\text{{id}}}
\newcommand{\im}{\text{im}\,}
\newcommand{\codim} {\text{codim}\,}
\newcommand{\resol}{\text{resol}\,}
\newcommand{\rank}{\text{rank}\,}
\newcommand{\lpd}{\text{lpd}\,}
\newcommand{\coker}{\text{coker}\,}
\newcommand{\supp}{\text{supp}\,}
\newcommand{\Ad}{A_\cdot}
\newcommand{\Bd}{B_\cdot}
\newcommand{\Fd}{F_\cdot}
\newcommand{\Gd}{G_\cdot}


\newcommand{\sus}{\subseteq}
\newcommand{\sups}{\supseteq}
\newcommand{\pil}{\rightarrow}
\newcommand{\vpil}{\leftarrow}
\newcommand{\rpil}{\leftarrow}
\newcommand{\lpil}{\longrightarrow}
\newcommand{\inpil}{\hookrightarrow}
\newcommand{\pils}{\twoheadrightarrow}
\newcommand{\projpil}{\dashrightarrow}
\newcommand{\dotpil}{\dashrightarrow}
\newcommand{\adj}[2]{\overset{#1}{\underset{#2}{\rightleftarrows}}}
\newcommand{\mto}[1]{\stackrel{#1}\longrightarrow}
\newcommand{\vmto}[1]{\stackrel{#1}\longleftarrow}
\newcommand{\mtoelm}[1]{\stackrel{#1}\mapsto}
\newcommand{\bihom}[2]{\overset{#1}{\underset{#2}{\rightleftarrows}}}
\newcommand{\eqv}{\Leftrightarrow}
\newcommand{\impl}{\Rightarrow}

\newcommand{\iso}{\cong}
\newcommand{\te}{\otimes}
\newcommand{\into}[1]{\hookrightarrow{#1}}
\newcommand{\ekv}{\Leftrightarrow}
\newcommand{\equi}{\simeq}
\newcommand{\isopil}{\overset{\cong}{\lpil}}
\newcommand{\equipil}{\overset{\equi}{\lpil}}
\newcommand{\ispil}{\isopil}
\newcommand{\vvi}{\langle}
\newcommand{\hvi}{\rangle}
\newcommand{\susneq}{\subsetneq}
\newcommand{\sgn}{\text{sign}}


\newcommand{\xd}{\check{x}}
\newcommand{\ortog}{\bot}
\newcommand{\tL}{\tilde{L}}
\newcommand{\tM}{\tilde{M}}
\newcommand{\tH}{\tilde{H}}
\newcommand{\tvH}{\widetilde{H}}
\newcommand{\tvh}{\widetilde{h}}
\newcommand{\tV}{\tilde{V}}
\newcommand{\tS}{\tilde{S}}
\newcommand{\tT}{\tilde{T}}
\newcommand{\tR}{\tilde{R}}
\newcommand{\tf}{\tilde{f}}
\newcommand{\ts}{\tilde{s}}
\newcommand{\tp}{\tilde{p}}
\newcommand{\tr}{\tilde{r}}
\newcommand{\tfst}{\tilde{f}_*}
\newcommand{\empt}{\emptyset}
\newcommand{\bfa}{{\mathbf a}}
\newcommand{\bfb}{{\mathbf b}}
\newcommand{\bfd}{{\mathbf d}}
\newcommand{\bfl}{{\mathbf \ell}}
\newcommand{\bfx}{{\mathbf x}}
\newcommand{\bfm}{{\mathbf m}}
\newcommand{\bfv}{{\mathbf v}}
\newcommand{\bft}{{\mathbf t}}
\newcommand{\bbfa}{{\mathbf a}^\prime}
\newcommand{\la}{\lambda}
\newcommand{\bfen}{{\mathbf 1}}
\newcommand{\bfe}{{\mathbf 1}}
\newcommand{\ep}{\epsilon}
\newcommand{\en}{r}
\newcommand{\tu}{s}
\newcommand{\Sym}{\text{Sym}}

\newcommand{\ome}{\omega_E}

\newcommand{\bevis}{{\bf Proof. }}
\newcommand{\demofin}{\qed \vskip 3.5mm}
\newcommand{\nyp}[1]{\noindent {\bf (#1)}}
\newcommand{\demo}{{\it Proof. }}
\newcommand{\demodone}{\demofin}
\newcommand{\parg}{{\vskip 2mm \addtocounter{theorem}{1}  
                   \noindent {\bf \thetheorem .} \hskip 1.5mm }}

\newcommand{\lcm}{{\text{lcm}}}


\newcommand{\dl}{\Delta}
\newcommand{\cdel}{{C\Delta}}
\newcommand{\cdelp}{{C\Delta^{\prime}}}
\newcommand{\dlst}{\Delta^*}
\newcommand{\Sdl}{{\mathcal S}_{\dl}}
\newcommand{\lk}{\text{lk}}
\newcommand{\lkd}{\lk_\Delta}
\newcommand{\lkp}[2]{\lk_{#1} {#2}}
\newcommand{\del}{\Delta}
\newcommand{\delr}{\Delta_{-R}}
\newcommand{\dd}{{\dim \del}}
\newcommand{\Del}{\Delta}

\newcommand{\bb}{{\bf b}}
\newcommand{\cc}{{\bf c}}
\newcommand{\xx}{{\bf x}}
\newcommand{\yy}{{\bf y}}
\newcommand{\zz}{{\bf z}}
\newcommand{\mv}{{\xx^{\aa_v}}}
\newcommand{\mF}{{\xx^{\aa_F}}}

\newcommand{\Symm}{\text{Sym}}
\newcommand{\pnm}{{\bf P}^{n-1}}
\newcommand{\opnm}{{\go_{\pnm}}}
\newcommand{\ompnm}{\omega_{\pnm}}

\newcommand{\pn}{{\bf P}^n}
\newcommand{\hele}{{\mathbb Z}}
\newcommand{\nat}{{\mathbb N}}
\newcommand{\rasj}{{\mathbb Q}}
\newcommand{\bfone}{{\mathbf 1}}

\newcommand{\dt}{\bullet}
\newcommand{\disk}{\scriptscriptstyle{\bullet}}

\newcommand{\cxF}{F_\dt}
\newcommand{\pol}{f}

\newcommand{\Rn}{{\mathbb R}^n}
\newcommand{\An}{{\mathbb A}^n}
\newcommand{\frg}{\mathfrak{g}}
\newcommand{\PW}{{\mathbb P}(W)}

\newcommand{\pos}{{\mathcal Pos}}
\newcommand{\g}{{\gamma}}

\newcommand{\Vaa}{V_0}
\newcommand{\Bp}{B^\prime}
\newcommand{\Bpp}{B^{\prime \prime}}
\newcommand{\bbp}{\mathbf{b}^\prime}
\newcommand{\bbpp}{\mathbf{b}^{\prime \prime}}
\newcommand{\bp}{{b}^\prime}
\newcommand{\bpp}{{b}^{\prime \prime}}

\newcommand{\oLa}{\overline{\Lambda}}
\newcommand{\ov}[1]{\overline{#1}}
\newcommand{\ovv}[1]{\overline{\overline{#1}}}
\newcommand{\tm}{\tilde{m}}
\newcommand{\po}{\bullet}

\newcommand{\surj}[1]{\overset{#1}{\twoheadrightarrow}}
\newcommand{\Supp}{\text{Supp}}

\def\CC{{\mathbb C}}
\def\GG{{\mathbb G}}
\def\ZZ{{\mathbb Z}}
\def\NN{{\mathbb N}}
\def\RR{{\mathbb R}}
\def\OO{{\mathbb O}}
\def\QQ{{\mathbb Q}}
\def\VV{{\mathbb V}}
\def\PP{{\mathbb P}}
\def\EE{{\mathbb E}}
\def\FF{{\mathbb F}}
\def\AA{{\mathbb A}}

\newcommand{\oR}{\overline{R}}
\newcommand{\bfu}{{\mathbf u}}
\newcommand{\nn}{{\mathbf n}}
\newcommand{\oa}{\overline{a}}
\newcommand{\cop}{\text{cop}}
\renewcommand{\op}{{{\text{op}}}}
\renewcommand{\mm}{{\mathbf m}}
\newcommand{\ngmi}{\text{neg}}
\newcommand{\up}{\text{up}}
\newcommand{\dw}{\text{down}}
\newcommand{\di}[1]{\hat{#1}}
\newcommand{\diw}[1]{\widehat{#1}}
\newcommand{\bo}{b}
\newcommand{\ub}{u}
\newcommand{\fs}{\infty}
\newcommand{\ifst}{\infty}
\newcommand{\mon}{{mon}}
\newcommand{\cl}{\text{cl}}
\newcommand{\intr}{\text{int}}
\newcommand{\ul}[1]{\underline{#1}}
\renewcommand{\ov}[1]{\overline{#1}}
\newcommand{\bipil}{\leftrightarrow}
\newcommand{\bfc}{{\mathbf c}}
\renewcommand{\mp}{m^\prime}
\newcommand{\np}{n^\prime}
\newcommand{\Mod}{\text{Mod }}
\newcommand{\Sh}{\text{Sh } }
\newcommand{\st}{\text{st}}
\newcommand{\hM}{\tilde{M}}
\newcommand{\hs}{\tilde{s}}
\newcommand{\ee}{\mathbf{e}}
\renewcommand{\dd}{\mathbf{d}}
\renewcommand{\en}{{\mathbf 1}}
\long\def\ignore#1{}
\newcommand{\lex}{{\text{lex}}}
\newcommand{\ordGL}{\succeq_{\lex}}
\newcommand{\ordG}{\succ_{\lex}}
\newcommand{\ordML}{\preceq_{\lex}}
\newcommand{\ordM}{\prec_{\lex}}
\newcommand{\tLa}{\tilde{\Lambda}}
\newcommand{\tGa}{\tilde{\Gamma}}
\newcommand{\STS}{\text{STS}}
\newcommand{\jj}{\mathbf{j}}
\renewcommand{\mod}{\text{ mod}\,}
\newcommand{\shmod}{\texttt{shmod}\,}
\newcommand{\hf}{\underline{f}}
\newcommand{\Glim}{\lim}
\newcommand{\Gcolim}{\colim}
\newcommand{\fm}{f^{\underline{m}}}
\newcommand{\fn}{f^{\underline{n}}}
\newcommand{\gn}{g_{\mathbf{|}n}}
\newcommand{\se}[1]{\overline{#1}}
\newcommand{\uu}{{!}}
\newcommand{\ii}{{\scalebox{0.78}{\rotatebox[origin=c]{180}{\rm{!}}}}}
\newcommand{\iig}{{\rotatebox[origin=c]{180}{\rm{!}}}}
\newcommand{\hhat}[1] {\hat{\hat{#1}}}
\renewcommand{\adj}{\dashv}
\renewcommand{\susneq}{\not \subseteq}
\newcommand{\Set}{\text{\bf Set}}
\newcommand{\con}{\text{con}}
\newcommand{\lrarr}{\, \Leftrightarrow \, }
\newcommand{\hhP}{\hhat{P}}
\newcommand{\hhA}{\hhat{A}}
\newcommand{\hhQ}{\hhat{Q}}
\newcommand{\hhB}{\hhat{B}}
\newcommand{\dhat}{\widehat{\widehat{}}}

\newcommand{\hP}{\hat{P}}
\newcommand{\hA}{\hat{A}}
\newcommand{\hQ}{\hat{Q}}
\newcommand{\hB}{\hat{B}}

\newcommand{\bfto}{{\mathbf 2}}
\renewcommand{\cD}{{\mathcal D}}
\newcommand{\cset}{{\mathbf {Set}}}

\newcommand{\ad}{\text{ad}}
\newcommand{\CMon}{{\mathbf{CMon}}}
\renewcommand{\kr}{{\mathbf k}}
\newcommand{\kCAlg}{{\mathbf{CAlg}_{\kr}}}
\newcommand{\cone}{{\text{cone}}}
\newcommand{\core}{{\text{{\it core}}}}
\newcommand{\pdel}{{\vardel \Delta}}
\newcommand{\ojoin}{{\bullet}}

\newcommand{\kring}{\kCAlg}
\newcommand{\scu}{{\mathbf {sc_1}}}
\newcommand{\sct}{{\mathbf {sc_2}}}
\newcommand{\sco}{{\mathbf {sc_0}}}
\newcommand{\scA}{{{\mathbf {sc}}_A}}
\newcommand{\scB}{{{\mathbf {sc}}_B}}

\newcommand{\SCu}{{\mathbf {SC_1}}}
\newcommand{\SCt}{{\mathbf {SC_2}}}
\newcommand{\SCo}{{\mathbf {SC_0}}}
\newcommand{\SCA}{{{\mathbf {SC}}_A}}
\newcommand{\SCB}{{{\mathbf {SC}}_B}}

\newcommand{\mult}{{\scalebox{0.5}{$\bullet$}}}
\newcommand{\pluss}{{\scalebox{0.5}{$+$}}}

\newcommand{\sqr}{{\text{\bf sqfRing}}}

\makeatletter
\newcommand{\colim@}[2]{%
  \vtop{\m@th\ialign{##\cr
    \hfil$#1\operator@font colim$\hfil\cr
    \noalign{\nointerlineskip\kern1.5\ex@}#2\cr
    \noalign{\nointerlineskip\kern-\ex@}\cr}}%
}
\newcommand{\colim}{%
  \mathop{\mathpalette\colim@{\rightarrowfill@\textstyle}}\nmlimits@
}
\makeatother

\makeatletter
\newcommand{\lim@}[2]{%
  \vtop{\m@th\ialign{##\cr
    \hfil$#1\operator@font lim$\hfil\cr
    \noalign{\nointerlineskip\kern1.5\ex@}#2\cr
    \noalign{\nointerlineskip\kern-\ex@}\cr}}%
}
\newcommand{\limi}{%
  \mathop{\mathpalette\lim@{\leftarrowfill@\textstyle}}\nmlimits@
}
\makeatother

\begin{abstract}
  For a set $A$ let $\SCA$ be the poset of simplicial complexes whose
  vertices are in $A$. For a function $f : A \pil B$ there
  are functors
  \[ f^{\uu \uu}, f^{**}, f^{\ii \ii} : \SCA \pil \SCB, \quad
f^{\uu *}, f^{\ii*} : \SCB \pil \SCA, \] 
forming a five sequence of adjoints
 \[ f^{\uu \uu} \adj f^{* \uu} \adj f^{* *} \adj f^{* \ii} \adj f^{\ii \ii}. \]
 We investigate in detail these functors, and
 use this to give three categorical structures on simplicial complexes
 on finite sets such that the Stanley-Reisner correspondence
to commutative monomial rings gives dualities. 
\end{abstract}
\maketitle

\section*{Introduction}

A (combinatorial) simplicial complex $X$ on a set $A$ is a family of
subsets of $A$ such that if $F \in X$ and $G \sus F$ then $G \in X$.
Let $\SCA$ be the poset (ordered by inclusion) of simplicial
complexes on $A$. Given a function $f : A \pil B$, 
there are functors (order-preserving maps)
\[ f^{\uu \uu}, f^{**}, f^{\ii \ii} : \SCA \pil \SCB, \quad
f^{\uu *}, f^{\ii*} : \SCB \pil \SCA, \]
such that these form a five sequence of adjoint functors
\[ f^{\uu \uu} \adj f^{* \uu} \adj f^{* *} \adj f^{* \ii} \adj f^{\ii \ii}. \]
We investigate these functors, and how they transform
simplicial complexes. In particular we do this in detail for
$f$ injective and for $f$ surjective.

\medskip
Let $k$ be a commutative ring (usually a field), and $k[x_A]$ the
polynomial ring with $x_a, a \in A$ as variables. To a simplicial complex
$X$ on a finite set $A$ is associated the {\it face ring} or {\it Stanley-Reisner ring}
$k[x_A]/I_X$ where $I_X \sus k[x_A]$ is the ideal generated by monomials
$\prod_{i \in N} x_i$ where $N \sus A$ ranges over the subsets not in $X$.

Since their introduction in the mid-seventies,
Stanley-Reisner rings have been massively studied, see standard
textbooks \cite{HeHi, MiSt, Stanley, Vill}. The important
notions of Cohen-Macaulay
and Gorenstein simplicial complexes are defined via their Stanley-Reisner
rings. Major achievements of this correspondence was R.Stanley's
proof in 1975 of the Upper Bound Conjecture for simplicial spheres
\cite{St1},
and some years later the characterization of $f$-vectors of simplicial
polytopes, \cite{BillLee, St2}.

Normally in mathematics a
correspondence between geometry and algebra is functorial. This also
holds for the Stanley-Reisner correspondence. But this has received little
(if any) attention, perhaps due to a certain unnaturality: the morphisms of
rings do not respect multigradings.
Using the above functors we give two new categorical structures on
simplicial complexes on finite sets, and define three categorical
structures on Stanley-Reisner rings. We get three dualities
between categories of simplicial complexes and categories of
Stanley-Reisner rings. The two new dualities introduced respect multigradings,
as the ring homomorphisms send homogeneous elements to homogeneous elements.

Let us mention another recent direction, \cite{SS}, considering
simplicial complexes and Stanley-Reisner theory
from a more principled setting. It
relates to algebraic geometry and sheaf theory. In particular
the set of subsets of a finite $n$-set is considered as the affine scheme
${\mathbb A}^n_{{\mathsf F}_1}$ over the ``field with one element''.
An early contribution in this direction is \cite{Ya}.

The organization of this article is as follows. Section \ref{sec:poset}
contains preliminaries on posets, Galois correspondences and adjoint
functors. Section \ref{sec:sicx} details the notion of simplicial complexes
and establishes the five-sequence of adjoint functors. Section \ref{sec:AB}
gives explicit descriptions of these functors. Section \ref{sec:functSR}
gives two new categorical structures on simplicial complexes making
the Stanley-Reisner correspondence functorial.

In Section \ref{sec:alexD} we show how the functors interact with
Alexander duality. The short Section \ref{sec:term} recalls some
terminology for simplicial complexes. Section \ref{sec:inj} gives more
detail on the functors when $f : A \pil B$ is injective, and Section
\ref{sec:surj} does the same when $f : A \pil B$ is surjective.
In the last Section \ref{sec:prod} we use the functors to define
several products of simplicial complexes $X, Y$ on disjoint
sets $A$ and $B$.

\medskip
\noindent {\it Note.} Simplicial complexes
are assumed to be defined on arbitrary
sets, usually denoted $A$ or $B$. But in statements related to
polynomial rings $k[x_A]$ or $k[y_B]$, we always assume $A$ and $B$ to
be finite.

\section{Preliminaries on posets}
\label{sec:poset}
Posets give rise to distributive lattices, and we recall this. We
also see how an order-preserving morphism between posets $f : P \pil Q$,
gives rise to three adjoint functors $f^\uu \adj f^* \adj f^\ii$
between the associated distributive lattices. 

\subsection{Galois correspondences
  and minimal and maximal elements}
\label{subsec:adj-maxmin}
Given posets $P,Q$. Recall that a Galois correspondence between
them is two order-preserving maps 
 $P \bihom{F}{G} Q$ such that $F(p) \leq q$ if and only if
 $p \leq G(q)$. 
The following is immediate:
\begin{lemma} \label{lem:poset-maxmin} Given such a Galois correspondence.
  Let $p_0 \in P$ and $q_0 \in Q$.
  \begin{itemize}
  \item[a.] The unique maximal $p \in P$ such that $F(p) \leq q_0$,
    is $p = G(q_0)$.
  \item[b.] The unique minimal $q \in Q$ such
    that $G(q) \geq p_0$ is $q = F(p_0)$.
  \end{itemize}
\end{lemma}

A Galois correspondence is a special case of the much more
general notion of an adjunction between categories.
We will use the adjunction terminology instead of Galois correspondence.

\subsection{The lattice of cuts in a posets}
\label{subsec:ds-adj}
For posets $P,Q$, denote by $\Hom(P,Q)$ the set of order-preserving
maps between them. This is itself a poset by $\phi \leq \psi$ if
$\phi(p) \leq \psi(p)$ for every $p \in P$. Denote $\bfto := \{0 < 1\}$. 

A {\it down-set} $D$ in a poset $P$
is a subset of $P$ such that if $q \in D$ and $p \leq q$, then $p \in D$.
For a poset $P$, denote by $P^{\op}$
the {\it opposite poset} with order $p \leq_{\op} q$ if $p \geq q$.
The lattice of down-sets of $P$ identifies as $\Hom(P^{\op}, \bfto)$.
A down-set $D$ gives a map
\[ \phi : P^{\op} \pil \bfto, \quad \phi(p) = \begin{cases} 1, & p \in D \\
0, & p \not \in D \end{cases}. \]
For short, denote by $\hP$ the set of order-preserving maps
$\Hom(P^{\op}, \bfto)$,
which we equally well identify as the down-sets of $P$.
It is the distributive lattice associated to $P$. (Birkhoff's
representation theorem says that any distributive lattice is of this
form, \cite{PD}.)

If $D$ is a down-set, the complement $U$ of $D$ in $P$ is an {\it up-set}.
We may consider the elements of $\hP$ to be {\it cuts} $(D,U)$.
We sometimes write this as $(D,-)$ or $(-,U)$ according to which part
we focus on.

\medskip
Given an order-preserving map $f : P \pil Q$. This induces an
order-preserving map
\[ f^* : \Hom(Q^{\op}, \bfto) \pil \Hom(P^{\op}, \bfto), \,
\text{ or for short } f^* : \hQ \pil \hP, \] 
sending a cut $(D,U)$ of $Q$ to the cut $(f^{-1}(D), f^{-1}(U)$ of $P$.
This map $f^*$ has both a left and a right adjoint
\[ f^\uu, f^\ii : \hP \pil \hQ, \quad f^\uu \adj f^* \adj f^\ii, \]
as we now explain. 
For a subset $S$ of the poset $P$, denote by $S^{\downarrow}$ the down-set
of $P$ generated by  $S$, so
\[ S^{\downarrow} = \{ p \in P \, | \, p \leq s \text{ for some } s \in S \}.
\]
Similarly we have $S^{\uparrow}$, the up-set generated by $S$.
The adjoints are then given by
\[ f^\uu(D,U) = (f(D)^{\downarrow}, -), \quad
f^\ii(D,U) = (-, f(U)^{\uparrow}). \]

\begin{example}
Consider $A$ to be a set of apples and $B$ buckets. The function $f : A \pil B$
puts apples into buckets. Then:
\begin{itemize}
\item $f^* : \hB \pil \hA$ sends a set of buckets
$C \sus B$ to the set of apples contained in these buckets, $f^{-1}(C) \sus A$.
\item $f^\uu : \hA \pil \hB$ sends a subset of apples $D \sus A$ to the
buckets containing {\it some} element from $D$ (the ``exist-a-$D$ apple'' buckets)
\item $f^\ii : \hA \pil \hB$ sends the subset $D$ to the buckets
which {\it only} contain apples from $D$ (the ``all-$D$ apple'' buckets)
\end{itemize}
\end{example}

Given a set map $f : A \pil B$ and $D \sus A$, the {\it $f$-core} of $D$
is the largest $C \sus B$ such that $f^{-1}(C) \sus D$. We write
$\core_f (D)$ for this $C$.
Note that $\core_f(D)$ is the set of all $b$ such that the full fiber
$f^{-1}(b)$ is contained in $D$. In particular if $E = B \backslash f(A)$,
then $E$ is contained in every $\core_f(D)$. 
For $(D,U) \in \hA$ we observe that:
\begin{itemize}
\item $f^\uu(D,U) = (f(D), \core_f(U))$,
\item $f^\ii(D,U) = (\core_f(D), f(U))$.
\end{itemize}


\medskip    
We get the following from Lemma \ref{lem:poset-maxmin}
concerning fibers when we have both
a left and right adjoint. 
\begin{corollary} \label{lem:adj-fiber} Given adjoints
  $f^\uu \adj f^* \adj f^\ii$ of posets:
  \[ f^\uu, f^\ii : P \pil Q, \quad f^* :  Q \pil P. \]
  Suppose $f^*(q) = p$, then such $q$ giving value $p$,
  are precisely those in the interval
  $f^\uu(p) \leq q \leq f^\ii(p)$ (note that the interval may be empty). 
\end{corollary}

\begin{proof}
  If $f^*(q) = p$, by Lemma \ref{lem:poset-maxmin}
  $f^\uu(p) \leq q \leq f^\ii(p)$. Furthermore  
  \[ p \leq  f^* \circ f^{\uu}(p) \leq f^*(q) \leq f^* \circ f^{\ii}(p) \leq p, \]
  so all give value $p$. 
\end{proof}

\subsection{The general setting}
Before proceeding we note how this fits in a very general setting.
For categories $\cC$ and $\cD$ denote by $\Hom(\cC, \cD)$ the functors
from $\cC$ to $\cD$. This becomes itself a category with morphisms
between functors $F$ and $G$ the natural transformations $F \Rightarrow G$.
Denote by $\cset$ the category of sets. The category of presheaves on
$\cC$ has objects the functors $\Hom(\cC^{\op}, \cset)$.
Given a functor $F : \cC \pil \cD$, let $\cZ$ be another category.
This induces a functor
\[ F^* : \Hom(\cD^{\op}, \cZ) \pil \Hom(\cC^{\op}, \cZ). \]
One may ask if this functor has a left or right adjoint. Adjoints
in this setting are called Kan extensions. We have the following
\cite[Cor.6.2.6]{Riehl}.

\begin{proposition} \label{pro:adj-existence}
  Suppose $\cC$ and $\cD$ are small categories and
  $F : \cC \pil \cD$ a functor. If $\cZ$ is complete and co-complete
  (i.e. has all limits and colimits), then
  $F^* : (\cD^{\op}, \cZ) \pil (\cC^{\op}, \cZ)$ has a left adjoint
  $F^\uu$ and a right adjoint $F^\ii$.
\end{proposition}

Note that $\cset$ is complete and co-complete. We are applying this
in the setting of enriched categories, \cite{Kelly} or \cite[Sec.4.4]{ACT},
where $\cset$ is replaced by the category $\bfto$.
A poset $P$ is then a $\bfto$-category,
since if $p,q \in P$, then $\Hom(p,q)$ is an object
in $\bfto$: This Hom-set is $1$ if $p \leq q$ and $0$ if we do not have
$p \leq q$. Generally, in an enriched category $\cC$, one
replaces
$\Set$ with a symmetric monoidal closed category $\cM$ such that for objects
$P,Q$ in $\cC$ the morphisms $\Hom(P,Q)$ is an object in $\cM$.

\section{Simplicial complexes and the five-sequence
  of adjoints}
\label{sec:sicx}

We develop the conceptual setting of (combinatorial) simplicial complexes
on a set, and we give the five-sequence of adjoints.

\subsection{Simplicial complexes}
\label{subsec:adj-term}
For a set $A$ let $P(A)$ be the power set of $A$, the set of
   all subsets of $A$. This is a (Boolean) lattice ordered by
   inclusion of subsets. A {\it (combinatorial) simplicial complex} on $A$ is a set $X$ 
   of subsets of $A$ such that if $F \in X$ and $G \sus F$, then $G \in X$.
   Such $X$ are precisely sets of subsets of $P(A)$ closed under taking smaller
   subsets, or a {\it down-set} of $P(A)$. For basics on combinatorial
   simplicial complexes, we refer to \cite{HeHi, MiSt, Stanley}.

 The power set $P(A)$ may be identified with maps (set functions)
 $\phi : A \pil \bfto$. A subset $D \sus A$ identifies with the function
 \[ \phi : A \pil \bfto, \quad \phi(a) = \begin{cases} 1, & a \in D \\
     0, & a \not \in D \end{cases}. \]
 Thus subsets $D \sus A$ corresponds to elements of
 \[ \hA = \Hom(A^{\op},\bfto) = \Hom(A, \bfto). \]
 Iterating the construction of Subsection \ref{subsec:ds-adj}, for
a poset $P$:
\[ \hhat{P} = \Hom((\hat{P})^\op,\bfto) = \Hom(\widehat{P^\op}, \bfto). \]
An element here is a cut $(\cD, \cU)$ where $\cD$ is a down-set
of $\hat{P}$ and $\cU$ is the complementary up-set.
A simplicial complex $X$ on $A$, corresponds to a down-set
in $\hA$. These correspond again precisely to order-preserving maps
($D$ is a subset of $A$)
\[ \Psi : (\hA)^{\op} \pil \bfto, \quad \phi(D) = \begin{cases} 1, & D \in X \\
 0, & D \not \in X \end{cases}. \]
Thus simplicial complexes correspond precisely to elements of 
$ \hhat{A} = \Hom((\hA)^{\op}, \bfto)$.

An element in $\hhA$ is a cut $(\cD, \cU)$, 
where $\cD$ is a down-set in the boolean
lattice $\hat{A}$ which identifies as the Boolean lattice $P(A)$
of subsets of $A$. Hence elements of $\hhA$ identifies as a
simplicial complex $X$.
Precisely this is as follows:
The elements of $\cD$ are cuts $(D,U)$. The elements of $X$ are
then the subsets $D$ of $A$. We shall henceforth
use $\hhA$ as representing simplicial complexes $\SCA$.
It will be convenient, especially in
Section \ref{sec:prod}, to have a full
arsenal of terminology:

\begin{itemize}
\item[1.] When $(D,U) \in \cD$, then $D$ is a {\it face} of $X$, and
  $U$ is a {\it face complement } of $X$. 
 \item[2.] When $(D,U) \in \cU$, then $D$ is a {\it non-face} of $X$
   and $U$ is a
   {\it non-face complement} of $X$.
\end{itemize}

The dimension of a face is $D$ is $|D|-1$. The dimension of $X$
is the maximal dimension of a face.

\subsection{The five-sequence of adjoints}
Now consider an order-preserving map $f : P \pil Q$ of posets. By Proposition
\ref{pro:adj-existence} we get left and right adjoints
\begin{equation} \label{eq:adj-ppp}
  f^{\uu}, f^\ii : \hat{P} \bihom{}{} \hat{Q} : f^*.
\end{equation}
 We may repeat this process and get further maps between
$\hhat{P}$ and $\hhat{Q}$. If $p,q \in \{ \uu, *, \iig\}$ denote
$(f^p)^q$ as $f^{pq}$. Each of the maps $f^p$ in \eqref{eq:adj-ppp}
gives three maps $f^{p\uu}, f^{p*}$ and $f^{p\ii}$. In all one gets
nine such maps, as $p = \uu, * ,\iig$, but not all of these are distinct.

\begin{proposition} Let $f : P \pil Q$ be order-preserving maps of posets.
  Then
  \[ \text{a. } \, f^{\uu *} = f^{* \uu}, \quad \text{b. } \,
    f^{\ii *} = f^{* \ii }, \quad \text{c. } \,
    f^{* * } = f^{\ii \uu} = f^{\uu \ii}. \]
  Thus we have three distinct functors
  \[ f^{\uu \uu}, f^{**}, f^{\ii \ii} : \hhat{P} \pil \hhat{Q}, \]
  and two distinct functors
  \[ f^{\uu *}, f^{\ii *} : \hhat{Q} \pil \hhat{P}. \]
\end{proposition}

\begin{proof}
  a. We have $f^\uu \adj f^*$. By \cite[Prop.4.4.6]{Riehl}
  $f^{\uu*} \adj f^{**}$. Since $f^{*\uu} \adj f^{**}$, by uniqueness
  of adjoints, \cite[Prop.4.4.1]{Riehl},
  we get $f^{*\uu} = f^{\uu *}$. Part b. is similar.

  c. We have $f^* \adj f^\ii$. Again by \cite[Prop.4.4.6]{Riehl}
  $f^{**} \adj f^{\ii *}$. Also $f^{\ii \uu} \adj f^{\ii *}$. Again by
  uniqueness of adjunctions $f^{\ii \uu} = f^{* * }$. Similarly
  $f^{\uu \ii} = f^{**}$.
\end{proof}

\begin{remark}
  This five-sequence of adjoints for posets is well-known to category
  theorists. It is described explicitly in the book \cite[Section 6]{Si}.
  But we are unaware of any detailed study of this five-sequence, at
  least from the perspective of simplicial complexes, as we now undertake.
\end{remark}

\begin{corollary} For a map $f : A \pil B$ of sets we get
a five-sequence of adjoints
\[ f^{\uu \uu} \adj f^{* \uu} \adj f^{* *} \adj f^{* \ii} \adj f^{\ii \ii}, \]
which are functors between
simplicial complexes on $A$ and $B$:
\begin{eqnarray} \label{eq:adj-AB}
  f^{\uu \uu}, f^{**}, f^{\ii \ii} & : & \hhat{A} \pil \hhat{B}\\
\label{eq:adj-BA}  f^{\uu *}, f^{\ii *} & : & \hhat{B} \pil \hhat{A}.
\end{eqnarray}
\end{corollary}

Our main objective is to describe and investigate these five functors.

\section{General statement and proofs}
\label{sec:AB}

We give the explicit description of the functors \eqref{eq:adj-AB}
and \eqref{eq:adj-BA}. We first note the immediate:

\begin{lemma} \label{lem:AB-X} Given a function $f : A \pil B$. 
  Let $X$ be a simplicial complex on $A$, and fix $C \sus B$.
  \begin{itemize}
       \item[a.] Every $D \sus A$ with $f(D) = C$ is a non-face of $X$ if
and only if no face $D$ of $X$ has $C \sus f(D)$.
    \item[b.] Every $D \sus A$ with $C = \core_f(D)$ is a face of $X$ if
and only if no non-face $D$ of $X$ has $C \geq \core_f(D)$.
    \end{itemize}
  \end{lemma}

  \begin{proof}
    Part a is obvious. As for part b, taking negations on each side
    the statement is that some non-face $D$ has $\core_f(D) = C$ iff
    some non-face $D$ has $C \geq \core_f(D)$, which is clear.
  \end{proof}

  \begin{lemma} \label{lem:AB-Y}
     Given a function $f : A \pil B$. 
  Let $Y$ be a simplicial complex on $B$, and let $D \sus A$.
  \begin{itemize}
  \item[a.] $f(D)$ is a face of $Y$ iff $D \sus f^{-1}(C)$
    for some face $C$ of $Y$.
  \item[b.] $\core_f(D)$ is a non-face of $Y$ iff $D \supseteq f^{-1}(C)$
    for some non-face $C$ of $Y$.
  \end{itemize}
  \end{lemma}

  \begin{proof}
    Both statements are immediate.
  \end{proof}

\begin{theorem} \label{thm:AB-fp*}
  Given a function $f : A \pil B$. 
  \begin{itemize}
  \item[a.]
    The map $f^{\uu *} : \hhat{B} \pil \hhat{A}$ sends $Y$ in $\hhat{B}$
    to the simplicial complex $X$ on $A$ whose faces are all $T \sus A$
    such that $f(T)$ is a face of $Y$. This is the largest $X$
    such that $f^{\uu \uu}(X) \sus Y$ and the smallest $X$ such
    that $f^{**}(X) \supseteq Y$.
   \item[b.] The map $f^{\ii *} : \hhat{B} \pil \hhat{A}$ sends $Y$
     in $\hhat{B}$ to the simplicial complex $X$ on $A$ whose
     faces are all $T \sus A$ such that $\core_f(T)$ is a face of $Y$.
     This is the largest $X$
    such that $f^{**}(X) \sus Y$ and the smallest $X$ such
    that $f^{\ii \ii}(X) \supseteq Y$.
   \end{itemize}
 \end{theorem}

  \begin{proof}
Let $X$ be given by $(\cD,\cU)$ in $\hhat{A}$, and
$Y$ be given by $(\cC,\cV)$ in $\hhat{B}$.

 a. Consider $f^* : \hat{B} \pil \hat{A}$. Then $\cD = f^{*\uu}(\cC)$ consists of
    those $(D,U)$ in  $\hat{A}$ such that for some $(C,V) \in \cC$ then
    $f^*(C) \supseteq D$, i.e. $f^{-1}(C) \supseteq D$.
    By Lemma \ref{lem:AB-Y}a,
    $X$ consists of those $D$ such that $f(D)$ is a face of $Y$.

\medskip
b. Consider $f^* : \hat{B} \pil \hat{A}$. Then $\cU = f^{*\ii}(\cV)$ consists of
    those $(D,U)$ in  $\hat{A}$ such that for some $(C,V) \in \cV$ then
    $f^*(C) \sus D$, i.e. $f^{-1}(C) \sus D$ for some non-face  $C$.
    
    By Lemma \ref{lem:AB-Y}b, the non-faces of  $X$ then consists of those $D$ such that
    $\core_f(D)$  is a non-face of $Y$. The faces of $X$ are then those
    $D$ such that $\core_f(D)$ is a face of $Y$. 

    The statement about smallest and largest follows by
    Lemma \ref{lem:poset-maxmin}.
  \end{proof}

  \begin{theorem} \label{thm:AB-fpp}
    Given a function $f : A \pil B$. 
    \begin{itemize}
    \item[a.] $f^{\uu \uu} : \hhat{A} \pil \hhat{B}$ sends $X$ on $A$
      to $Y$ on $B$ generated by the $f(D)$ where $D \in X$. It is
      the smallest $Y$ in $\hhat{B}$ such that $f^{* \uu}(Y) \supseteq X$.
     
    \item[b.] $f^{**} : \hhat{A} \pil \hhat{B}$ sends $X$ to $Y$ consisting
      of $C \sus B$ such that $f^{-1}(C) \in X$. Alternatively
      its faces are the $\core_f(D)$ of faces $D$ of $X$. It is the
      largest $Y$ on $B$ such that $f^{*\uu}(Y) \sus X$, and
      the smallest $Y$ such that $f^{* \ii}(Y) \supseteq X$. 

    \item[c.] $f^{\ii \ii} : \hhat{A} \pil \hhat{B}$ sends $X$
      to $Y$ consisting of all $C \sus B$ such that every $D \sus A$
      with $\core_f(D) = C$, is in $X$. It is the largest $Y$ on $B$
      such that $f^{* \ii}(Y) \sus X$. 
      \end{itemize}
    \end{theorem}

   \begin{proof} Again $X$ is given by $(\cD,\cU)$ in $\hhat{A}$, and
$Y$ by $(\cC,\cV)$ in $\hhat{B}$.

a. Consider $f^\uu : \hat{A} \pil \hat{B}$.
      Then $Y$ consists of those  $(C,-)$ in $\cC$ such that $f(D)
      \supseteq C$ for some $(D,-) \in \cD$.
      
      b. Consider $f^* : \hat{B} \pil \hat{A}$. Then 
      $(C,-) \in \cC$ if $(f^*(C),-) $ is in $\cD$. These
      are the $C \sus B$ such that $f^{-1}(C) \in X$.
    
      c. Consider $f^\ii : \hat{A} \pil \hat{B}$. Then $\cV$ consists
      of those $(C,V)$ such that $(C,V) \geq f^{\ii}(D,U) = (\core_f(D),f(U))$
      for some $(D,U) \in \cU$. Equivalently $C \geq \core_f(D)$ for some
      $(D,-)$ in $\cU$. 

Then $\cC$ consists of the $(C,V)$ such that $C \geq \core_f(D)$ for 
no non-face $D$. Lemma \ref{lem:AB-X}b gives that $Y$ consists
of those $C$ such that every $D$ with $\core_f(D) = C$ is a
face of $X$.    
\end{proof}

\section{Functoriality of Stanley-Reisner correspondence}
\label{sec:functSR}

To a combinatorial simplicial complex $X$ there is an associated ring,
its Stanley-Reisner ring. This is a central object in combinatorial
commutative algebra.
 In mathematics, natural and
 fruitful correspondences between geometry and algebra are normally functorial.
 The Stanley-Reisner correspondence is also functorial. But, perhaps due
 to a certain unnaturality (the morphisms on algebras do not respect
 multigradings), there has hardly been attention on this.

With the greater arsenal of maps
   between simplicial complexes we now have,
   it is possible to define three functorial correspondence between
   simplicial complexes and rings, giving dualities. And in fact two of the three do
   respect multigradings.
   
\subsection{Stanley-Reisner ring}
 For a finite set $A$ let $k[x_A]$ be the polynomial ring
whose variables are $x_a$ for $a \in A$. If $E \sus A$ write
$x^\mult_E$ for the monomial $\prod_{e \in E} x_e$ (called a square-free monomial
since it is a product of distinct variables).

When $X$ is a simplicial complex on the finite set $A$,
the {\it Stanley-Reisner ideal}
$I_X$ of $X$ is the ideal in $k[x_A]$ generated by monomials $x^\mult_E$
where $E$ is {\it not} in $X$. The {\it Stanley-Reisner} ring  is
$k[X] := k[x_A]/I_X$.
A squarefree monomial $x^\mult_E$ is then non-zero in $k[x_A]/I_X$ iff
$E \in X$. Note that the ring $k[X]$ will be $\hele^A$-graded.
   
\subsection{Categories of simplicial complexes}
Simplicial complexes can be made into a category. Given simplicial
complexes $X$ on $A$ and $Y$ on $B$ ($A,B$ not necessarily finite),
the traditional notion for a
morphism from $X$ to $Y$ is a function $f : A \pil B$ such that
for each face $F$ of $X$, the image $f(F)$ is a face of $Y$.
With our notation here, it says that a morphism from
$X$ to $Y$ is a function $f: A \pil B$ such $f^{\uu \uu}(X) \sus Y$.
Denote this category as $\SCo$.

We shall give two other ways to make simplicial complexes into a category.
\ignore{Note that a functorial
correspondence $F$ between geometry and algebra is normally contravariant:
A morphism between geometric objects $X \pil Y$ gives a morphism
$FY \pil FX$ in the opposite direction. This is also reasonable
to expect here: If $X \sus X^\prime$ is an inclusion of simplicial
complexes, defined on a set $A$, there is a surjection of
Stanley-Reisner rings $\kr[x_A]/I_{X^\prime} \pil \kr[x_A]/I_X$,
a surjection which is a $k$-algebra homomorphism.
}

\begin{proposition} \label{pro:funct-nf}
  Let $f : A \pil B$ be a set function, and $X$ and $Y$ simplicial
  complexes on respectively $A$ and $B$.
  \begin{itemize}
  \item[a.] If $Y = f^{**}(X)$, the map $f$ (or really $f^\uu$) sends
        non-faces of $X$ to non-faces of $Y$.
\item[b.] If $X = f^{* \ii}(Y)$, then $f^{-1}$ (or really $f^*$) sends
  non-faces of $Y$ to non-faces of $X$.
\end{itemize}
\end{proposition}

\begin{proof}
   For part a. let $N$ be a non-face of $X$. If  $f(N)$ was a face of $Y$,
   it would have to be a core of some face $F$ of $X$.
But this contradicts that $N$ is a non-face.

For part b. let $Y$ correspond to a cut $(\cC,\cV)$ of $\hat{B}$.
Then $f^{* \ii}(Y)$ is $(-, f^{* \uparrow}(\cV))$, and so non-faces of $Y$,
(elements of $\cV$) map to non-faces of $X$.
 \end{proof}

 In addition to $\SCo$,
we now define two other categories of simplicial complexes.

\begin{definition}  The categories $\SCu$ and $\SCt$ have
 simplicial complexes on finite sets as
  objects. Let  $X$ be a simplicial complex on the set $A$, and $Y$
  a simplicial complex on the set $B$.

  \medskip
  A morphism $F : X \pil Y$ in $\SCu$
  is a set map $f : A \pil B$ such that $f^{**}(X) \sus Y$, or equivalently
  (due to the adjunction $f^{**} \adj f^{* \ii}$) $X \sus f^{* \ii}(Y)$.

  \medskip
 A morphism $G : X \pil Y$ in $\SCt$
  is a set map $g : B \pil A$ such that $X \sus g^{**}(Y)$, or equivalently
  (due to the adjunction $g^{*\uu} \adj g^{**}$) $g^{*\uu}(X) \sus Y$.
\end{definition}

More explicitly a morphism $F : X \pil Y$ in $\SCu$ is a map $f : A \pil B$
such that:
\begin{itemize}
  \item If $C \sus B$ and  $f^{-1}(C)$ is a face of $X$,
  then $C$ is a face of $Y$.
\end{itemize}
If we think of $f : A \pil B$ as giving
a partition of the vertices of $A$ into color classes $A_b$, one for each
color $b \in B$, then
such a partition gives a morphism if whenever a union of color classes $A_{b_1} \cup \cdots \cup A_{b_m}$ is a face of $X$,
then $\{b_1, \ldots, b_m\}$ is a face of $Y$.

\medskip
On the other hand a morphism $G : X \pil Y$ in $\SCt$ given by $g : B \pil A$
is such that:
\begin{itemize}
\item If $D \sus A$ and $A$ is a face of $X$, then $g^{-1}(A)$ is
  a face of $Y$.
  \end{itemize}
We can think of $g$ as ``blowing up'' each vertex $a$ to a set $B_a$.
It gives a morphism if whenever $\{a_1, \ldots, a_n\}$ is a face of $X$,
then the ``blow up'' $B_{a_1} \cup \cdots \cup B_{a_n}$ is a face of $Y$. 

\begin{example}
  Let $X$ be the discrete simplicial complex on $A$, so that each vertex
  $a$ is a facet $\{a\}$. Then a morphism $G : X \pil Y$
  in $\SCt$ is equivalent to give a matching in $Y$, indexed by $A$,
  i.e. a partition
  $B = \cup_{a \in A} B_a$ such that each $B_a$ is a face of $Y$.
\end{example}

\begin{example} Let $Y = \{\emptyset \}$. Then a morphism $X \pil Y$
  in $\SCu$ is equivalent to give a partition $A = \cup_{b \in A} A_b$
  such that no monochromatic set $A_b$ is a face of $X$.
  See Example \ref{eks:surjU-join} for more detail.
  In particular, there are
  many morphisms $f : X \pil Y$ in $\SCu$ such that $X$ is non-empty
  and $Y = \{\emptyset \}$. So there is not a direct geometric
  interpretation of morphisms in $\SCu$. 
\end{example}

\subsection{Categories of rings}
Fix a commutative ring $k$. 
We define three categories of rings
\[ \sqr^0_k, \quad \sqr^1_k, \quad \sqr^2_k. \]

These categories all have the same {\it objects}: Quotients of
polynomial rings $k[x_A]/I$
where $A$ is a finite set and $I$ a squarefree monomial ideal.
So each generator of $I$
is a squarefree monomial $x_E^\mult$ where $E \sus A$. This is
a category of $k$-algebras but the objects come with additional data,
a set of variables.
The categories are distinguished by their {\it morphisms}:

\medskip
In each category a morphism is a $k$-algebra morphism
$k[y_B]/I_2 \pil k[x_A]/I_1$ induced by a $k$-algebra homomorphism
of polynomial rings (which is part of the data): $\phi: k[y_B] \pil k[x_A]$.
The homomorphism $\phi$ has the following form:
\begin{itemize}
\item  In $\sqr^2_k$, it 
  sends each variable $y_b$ to some variable $x_a$.
  So this morphism comes with a set map $g : B \pil A$, and $a = g(b)$.
\item In $\sqr^1_k$, it sends each squarefree monomial
  $y^\mult_C$ (with $C \sus B$) to a {\it squarefree} monomial
  $x^\mult_D$ (with $D \sus A$). Such a morphism is readily seen to
  come from a set map $f : A \pil B$,
and each variable $y_b$ is sent to $x^\mult_{A_b}$ (note that if $b$ is
not in the image of $f$, then $y_b \mapsto 1$). 
\item In $\sqr^0_k$, it sends each variable $y_b$ 
to a sum of variables $x_{a_1} + x_{a_2} + \cdots + x_{a_r}$ where
the sets $A_b = \{a_1, \ldots, a_r\}$ are disjoint for distinct $b$'s,
Again such a morphism comes from a map $f : A \pil B$,
and the variable $y_b$ is sent to $\sum_{a \in A_b} x_{a}$ (note that if $b$ is
not in the image of $f$, then $y_b \mapsto 0$).  
\end{itemize}

Note that in the last category $\sqr^0_k$, the maps do not
correspond to a homomorphism of multigradings $\ZZ^B \pil \ZZ^A$,
since a sum of two or more variables is not homogeneous.
However, morphisms in the
two other categories do send homogeneous elements to
homogeneous elements.

\subsection{Functors from simplicial complexes to
  commutative rings}

Now let $\scu$ be the subcategory of $\SCu$ consisting of
simplicial complexes $X$ on {\it finite} sets. Similarly we
have $\sco$ and $\sct$.
We define a functor: $\scu^{\op} \pil \sqr^1_k$ (the ${\op}$ signifies that
the functor is contravariant), by sending a simplicial complex
to its Stanley-Reisner ring. If $F : X \pil Y$ is a morphism in $\scu$,
let $X \sus X^\prime = f^{* \ii}(Y)$. We then get a $\kr$-algebra homomorphism:
 \[ \kr[y_B]/I_Y \pil \kr[x_A]/I_{X^\prime} \pil \kr[x_A]/I_X. \]
 The second map above  is due to $X$ being a subcomplex of
  $X^\prime$. The first map sends a variable $y_b$ to $x^\mult_{A_b}$. By
  Proposition \ref{pro:funct-nf}b. it is well defined.

  We define a functor: $\sct^{\op} \pil \sqr^2_k$,
by sending a simplicial complex
to its Stanley-Reisner ring. If $G : X \pil Y$ is a morphism in $\sct$,
let $X \sus X^\prime = g^{**}(Y)$. We then have maps
of Stanley-Reisner rings:
  \[ \kr[y_B]/I_Y \pil \kr[x_A]/I_{X^\prime} \pil \kr[x_A]/I_X. \] 
  The second map comes from the inclusion $X \sus X^\prime$. The first
  map sends a variable $y_b$ to $x_{f(b)}$.
  By Proposition \ref{pro:funct-nf}a. it is well defined
  (for its application switch $X$ and $Y$).

 We define a functor: $\sco^{\op} \pil \sqr^0_k$,
by sending a simplicial complex
to its Stanley-Reisner ring. If $F : X \pil Y$ is a morphism in $\sco$,
let $X \sus X^\prime = f^{*\uu}(Y)$. We then have maps
of Stanley-Reisner rings:
  \[ \kr[y_B]/I_Y \pil \kr[x_A]/I_{X^\prime} \pil \kr[x_A]/I_X. \] 
The second map comes from the inclusion $X \sus X^\prime$. The first
map sends a variable $y_b$ to the sum $\sum_{a \in A_b}x_{a}$.
This is well-defined, since if $N$ is a nonface of $Y$,
then for any choices  $a_n \in A_n$ for $n \in N$, the
set $E = \{a_n \, | \, n \in N \}$ is a nonface of $X$ (it cannot
be a face as $f(E)$ is a face of $Y$ for each face $E$ of $X$).

\medskip
\begin{theorem}
  The Stanley-Reisner correspondence gives dualities
  \[ \sco^{\op} \mto{\iso} \sqr^0_k,
    \quad \scu^{\op} \mto{\iso} \sqr^1_k,
  \quad \sct^{\op} \mto{\iso}\sqr^2_k\]
\end{theorem}

\begin{proof}
{\it Functoriality:}  Given morphisms in $\scu$:
  \[ X \mto{F_1} Y \mto{F_2} Z, \quad A \mto{f_1} B \mto{f_2} C. \]
  The maps on Stanley-Reisner rings are
  \[ k[z_C]/I_Z \pil k[y_B]/I_Y \pil k[x_A]/I_X, \quad z_c \mapsto y^\mult_{B_c}
    \mapsto \prod_{b \in B_c} x^\mult_{A_b}. \]
  But the latter is $x^\mult_{A_c}$ for the composition $A \mto{f_2 \circ f_1} C$.

  \medskip
  Given morphisms in $\sct$:
  \[ X \mto{G_1} Y \mto{G_2} Z, \quad A \vmto{g_2} B \vmto{g_1} C. \]
  The maps on Stanley-Reisner rings are
  \[ k[z_C]/I_Z \pil k[y_B]/I_Y \pil k[x_A]/I_X, \quad z_c \mapsto y_{g_1(c)}
    \mapsto x_{g_2 \circ g_1(c)}, \]
  which shows the functoriality.
The composition in the $\sco$-case is similar to the first.

  \medskip
\noindent{\it Duality:}  
Let $\phi: k[x_B]/I_2 \pil k[x_A]/I_1$ be
a homomorphism in one of the $\sqr_k$-categories,
and let $I_2$ and $I_1$ correspond to the
  simplicial complexes $Y$ and $X$ respectively.

\ignore{If $\phi \in \sqr^0_k$, then 
  when  $C \sus B$ is a nonface of $Y$, that $\phi$ is a homomorphism,
  gives that
  each $E \sus A$ with $f(E) = C$ will be a non-face of $X$. So whenever $E$
  is a face of $X$, then $f(E)$ is a face of $Y$. Then we have a morphism
  $F : X \pil Y$ in $\sco$.}

Let $\phi \in \sqr^1_k$. If $C \sus B$ is a non-face of $Y$,
since $\phi$ is a homomorphism, $f^{-1}(C)$ a non-face.
But then if $C \sus B$ is such
  that $f^{-1}(C)$ is a face, then $C$ must be face of $Y$. So
  $f^{**}(X) \sus Y$ and we have a morphism in $\scu$.

 If $\phi \in \sqr^2_k$, then 
   when  $C \sus B$ is a nonface of $Y$, $g(C)$ is a non-face of
   $X$. If $D \sus A$ is a face of $X$, then $C = g^{-1}(D)$ must be
   face, since otherwise $g(C) \sus D$ is a non-face.
   Thus we get a morphism $G : X \pil Y$ in $\sct$. 
  
   Now let  $\phi \in \sqr^0_k$.
   Let  $D \sus A$ be a face, and $C = f(D) \sus B$.
   Then there is a section $s : C \pil D$ of $f_{|D} : D \pil C$.
   So $s(C) \sus D$ and so $s(C)$ is
  face. If $C$ was a non-face then for every section $s$, $x^\mult_{s(C)}$
  is a term in $\phi(x^\mult_C)$ and so $s(C)$ would be a non-face, giving
  a contradiction. Thus $C = f(D)$ is a face and we have a morphism
  $F : X \pil Y$ in $\sco$.   
\end{proof}

\section{Alexander duality}
\label{sec:alexD}

This section is not needed for statements later, but is used briefly
in some of the arguments.
But Alexander duality is a basic duality for simplicial complexes,
and we show it 
interacts well with the functors we study.

\subsection{Alexander duality for posets}
For a poset $P$ there is an isomorphism
\begin{equation} \label{eq:AD-Piso} (\hat{P})^{\op} \overset{\tau_P}\lpil \widehat{(P^{\op})},
  \quad (D,U)^{\op} \mapsto
  (U^{\op}, D^{\op}).
  \end{equation}

This induces an order-reversing map
\[ \hat{P} \overset{\ad_P}{\lpil} \widehat{P^{\op}},
\quad (D,U) \mapsto (U^{\op}, D^{\op}) \]
called {\it Alexander duality} for posets.

Given an order-preserving map $f : P \pil Q$ we have
$f^{\op} : P^{\op} \pil Q^{\op}$.
It is straightforward to verify that we get commutative diagrams
\begin{equation} \label{eq:alexD-CD}
  \xymatrix{\hP \ar[r]^{f^{\uu}} \ar[d]_{\ad} & \hQ \ar[d]^{\ad} \\
    \widehat{P^{\op}} \ar[r]^{(f^{\op})^{\ii}} &  \widehat{Q^{\op}} } 
\qquad
\xymatrix{\hQ \ar[r]^{f^{*}} \ar[d]_{\ad} & \hP \ar[d]^{\ad} \\
\widehat{Q^{\op}} \ar[r]^{(f^{\op})^{*}} &  \widehat{P^{\op}} }
\end{equation}

\subsection{Alexander duality for simplicial complexes}

When the poset $P$ is a simply a set $A$, there is a natural isomorphism
$A^{\op} \overset{\iso}\lpil  A$ by $a^{\op} \mapsto a$. Thus we get
an order-reversing composite map
\[ \hA \overset{\ad_A}\lpil \widehat{A^{\op}} \overset{\iota_A} \lpil \hA, \quad
  (D,U) \mapsto (U^{\op}, D^{\op}) \mapsto (U,D).\]
For an element $(D,U)$ in $\hA$, denote:
\[ (D,U)^{\op} \in (\hA)^{\op}, \quad (D,U)^{\tilde{\op}}:= (U^{\op}, D^{\op})
  \in \widehat{A^{\op}}, \quad (D,U)^t:= (U,D) \in \hA. \]
For a subset ${\mathcal S} \sus \hA$ we get subsets
${\mathcal S}^{\op}, {\mathcal S}^{\tilde{\op}}, {\mathcal S}^t$ by
applying the operations to each of its elements.
We now get the following order-reversing composite,
Alexander duality for simplicial complexes:
\begin{equation*}  AD: \hhA \overset{\ad_{\hA}}\lpil {\widehat{(\hA)^{{\op}}}}
   \overset{{\hat{\tau}_A}} \lpil  \widehat {\widehat{(A^\op)}} 
 \overset{\hat{\iota}_A}\lpil \hhA,
\end{equation*}
where the maps are
\begin{equation*}
  (\cD, \cU) \mapsto (\cU^{\op}, \cD^{\op}) \mapsto
  (\cU^{\tilde{\op}}, \cD^{\tilde{\op}}) \mapsto (\cU^t, \cD^t).
\end{equation*}

\medskip
Given a set function $f : A \pil B$, we have $f^* : \hB \pil \hA$,
and the following commutes:
\begin{equation} \label{eq:alexD-enhatt}
  \xymatrix{\widehat{B^{\op}} \ar[d]^{\iota_B} \ar[r]_{(f^\op)^*}
& \widehat{A^{\op}} \ar[d]^{\iota_A}\\
\hB \ar[r]^{f^*} & \hA}, \quad
\xymatrix{\widehat{A^{\op}} \ar[d]^{\iota_A} \ar[r]_{(f^\op)^\uu}
& \widehat{B^{\op}} \ar[d]^{\iota_B}\\
\hA \ar[r]^{f^\ii} & \hB}.
\end{equation}
From \eqref{eq:alexD-CD} and \eqref{eq:alexD-enhatt} we get
commutative diagrams:
  \[ \xymatrix{
      \hhA \ar[d]_{\ad_{\hA}} \ar[r]^{{f^{**}}} &\hhB \ar[d]^{\ad_{\hB}} \\
       \widehat{(\hA)^{\op}} \ar[d]^{\hat{\tau}_A}
       \ar[r]^{{((f^{*})^\op)^*}} &  \widehat{\hB^{\op}} \ar[d]^{\hat{\tau}_A} \\
       \widehat{\widehat{A^{\op}}} \ar[d]^{\hat{\iota}_A} \ar[r]^{(f^{\op})^{* *}}
       & \widehat{\widehat{B^{\op}}}\ar[d]^{\hat{\iota}_A} \\
      \hhA \ar[r]^{{f^{**}}}  & \hhB }       
    \qquad
\xymatrix{
      \hhB \ar[d]_{\ad_{\hB}} \ar[r]^{{(f^{*})^\uu}} &\hhA \ar[d]^{\ad_{\hA}} \\
       \widehat{(\hB)^{\op}} \ar[d]^{\hat{\tau}_B}
       \ar[r]^{{((f^{*})^\op)^\ii}} &  \widehat{\hA^{\op}} \ar[d]^{\hat{\tau}_B} \\
       \widehat{\widehat{B^{\op}}} \ar[d]^{\hat{\iota}_B} \ar[r]^{(f^{\op})^{* \ii}}
       & \widehat{\widehat{A^{\op}}}\ar[d]^{\hat{\iota}_B} \\
      \hhB \ar[r]^{{ f^{* \ii}}}  & \hA }  
    \qquad
\xymatrix{
      \hhA \ar[d]_{\ad_{\hA}} \ar[r]^{{f^{\uu \uu}}} &\hhB \ar[d]^{\ad_{\hB}} \\
       \widehat{(\hA)^{\op}} \ar[d]^{\hat{\tau}_A}
       \ar[r]^{{((f^{\uu})^\op)^\ii}} &  \widehat{\hB^{\op}} \ar[d]^{\hat{\tau}_A} \\
       \widehat{\widehat{A^{\op}}} \ar[d]^{\hat{\iota}_A} \ar[r]^{(f^{\op})^{\ii \ii}}
       & \widehat{\widehat{B^{\op}}}\ar[d]^{\hat{\iota}_A} \\
      \hhA \ar[r]^{{f^{\ii \ii}}}  & \hhB }  
    \]

    For a simplicial complex $X$ write $X^\vee$ for its Alexander dual.
    The above gives:

    \begin{proposition} \label{pro:ad-ad}
      If $X$ (resp. $Y$ is a simplicial complex
      on $A$ (resp. $B$), denote by $X^\vee$ (resp. $Y^\vee$) its
      Alexander dual. Then:
      \[ f^{**}(X^\vee) = f^{**}(X)^\vee, \quad
         f^{\uu *}(Y^\vee) = f^{\ii *}(Y)^\vee, \quad
         f^{\uu \uu}(X^\vee) = f^{\ii \ii}(X)^\vee. \]
     \end{proposition}
     
\section{Terminology for simplicial complexes}
\label{sec:term}

We recall notions and terminology for simplicial complexes.
In the subsequent sections this enables us to get a more refined
understanding of the adjoint functors.

The top element in $\hhA$ corresponds to the simplex $\Delta(A)$,
consisting of all
subsets of $A$. The bottom element is the empty set $\emptyset$, containing
no subsets of $A$. (Note that it is uniquely covered by the simplicial complex
consisting of a single subset of $A$,
namely $\{\emptyset \}$.)
The boundary of the simplex $\Delta(A)$ is written $\pdel(A)$, and consists
of all subsets of $A$ save the full set $A$. (The simplex $\Delta(A)$
uniquely covers this boundary.)

For a simplicial complex $X$ corresponding to the cut $(\cD,\cU)$
of $\hat{A}$,
the maximal elements of $X$ (the $D$ in the cuts $(D,U)$ maximal in $\cD$)
are the {\it facets} of $X$.
For the minimal elements $(D,U)$ of $\cU$, the $D$ is usually referred to as
a {\it minimal non-face} of $X$. In the present article we
shall, for elegance of terminology and statements, refer to them
as the {\it co-facets} of $X$. 
Note that the minimal generators of the Stanley-Reisner ideal $I_X$
are precisely the $x^\mult_N$ where $N$ are the co-facets of $X$.

For $X$ a simplicial complex on $A$ and $E \sus A$, the {\it restriction}
$X_{|E}$ is the simplicial complex on $E$ consisting of all $F \sus E$ such that $F \in X$. The {\it link}
$lk_E(X)$ is the simplicial complex on $A \backslash E$ consisting
of all $F \sus A \backslash E$ such that
$F \cup E \in X$. In particular, if $E$ is not in $X$, the link
is the empty set $\emptyset$. 

If $Y$ is a simplicial complex on $B$ where $B$ is disjoint from $A$,
the {\it join} $X*Y$ consists of all $F \cup G$ with $F \in X$ and
$G \in Y$. The {\it cone} $\cone_B(X)$ is the join $X * \Delta(B)$.

\medskip
In the next two sections we
investigate in detail the functors \eqref{eq:adj-AB} and \eqref{eq:adj-BA},
and the associated changes of the Stanley-Reisner ring.
So $f : A \pil B$ is a function. It may be decomposed
into a surjection followed by an injection. We investigate these
cases separately. 

\section{Injections}
\label{sec:inj}

We identify an injection $f : A \pil B$
 with an inclusion $A \sus B$, and let $E = B \backslash A$.
To a simplicial complex $X$ on $A$, we write
$I_X$ for its associated squarefree monomial ideal in $k[x_A]$.
Similarly, for $Y$ on $B$ we write $I_Y$ for its ideal in $k[x_B]$.

\subsection{The functors $f^{\uu *}$ and $f^{\ii *}$}

\begin{proposition} \label{pro:inj-fp*}
Given an inclusion $f : A \sus B$.
  
\begin{itemize} 
\item[a.] The functor $f^{* \uu} : \hhat{B} \pil \hhat{A}$ sends 
$Y$ to its restriction $X = Y_{|A}$.
\item[b.] The functor $f^{*\ii} : \hhat{B} \pil \hhat{A}$ sends $Y$
  to the link $X = \text{lk}_E Y$.
\end{itemize}
\end{proposition}

\begin{corollary} \label{cor:inj-BA} Given an inclusion $f : A \sus B$.   
  \begin{itemize}
  \item[a.] The Stanley-Reisner ideal of $X = f^{*\uu} (Y)$ is
    $I_X = I_Y \cap k[x_A]$.
    \item[b.] The Stanley-Reisner ideal of $X = f^{*\ii} (Y)$ is
      $I_X = (I_Y:x_E) \cap k[x_A]$.
    \end{itemize} 
  \end{corollary}


  \begin{proof} This is immediate from Theorem \ref{thm:AB-fp*}.
    For b. note that $\core_f(D) = D \cup E$ for $D \sus A$. 
    \end{proof}

\begin{proof}[Proof of Corollary \ref{cor:inj-BA}]
Part a. is immediate from the definition of Stanley-Reisner ideal.

b. That a monomial $x^\mult_D$ is in the Stanley-Reisner ideal
$I_X$ means that $D$ is not in $X = \lk_E(Y)$. Equivalently $D \cup E$
is not in $Y$ and so $x^\mult_D \cdot x^\mult_E$ is in $I_Y$.
\end{proof}

\begin{remark} The restriction and link are very important notions for
  simplicial complexes. The minimal free resolution
  of the Stanley-Reisner ring $k[X]$ 
  as a module over $k[x_A]$ has multigraded betti numbers given by the
  reduced cohomology $\tilde{H}^i(X_{|A}, k)$ of the restrictions.
  This is a classical theorem of Hochster and
  a founding theorem of Stanley-Reisner theory, \cite{Hoch} and
  found in all basic textbooks \cite{HeHi, Peeva, MiSt, Stanley}.
  This may also be formulated in terms of links, \cite[Cor.8.1.4]{HeHi}.
  In particular the notion
  of $X$ being Cohen-Macaulay (derived from the algebraic notion of
  $k[X]$ being a Cohen-Macaulay ring) is that all the reduced homology groups
  $\tilde{H}_i(\lk_E X,k)$ of links  vanish for $i$ less than the dimension
  of $\lk_E X$, \cite{Reisner}.
\end{remark}

\subsection{The functors $f^{\uu \uu}, f^{**}, f^{\ii \ii}$}

For a set $E$, we consider variables $x_e$ indexed by $e \in E$.
We define the following ideals:
\begin{itemize}
\item $(x^+_E) = \sum_{e \in E} (x_e)$, the ideal generated by the variables
  $x_e$,
\item $(x^\mult_E) = \prod_{e \in E} (x_e)$, the ideal generated by the
  monomial $x^\mult_E$.
\end{itemize}

\begin{proposition}

\begin{itemize} Given an inclusion $f : A \sus B$. 
\item[a.] $f^{\uu \uu} : \hhat{A} \pil \hhat{B}$ sends $X$ to $Y$, which
  is $X$ considered
  as a simplicial complex on $B$ instead of $A$.
  So $f^{\uu \uu}(X)$ consists of all $C \sus B$
  such that $C \sus A$ and $C \in X$.
\item[b.] $f^{**} : \hhat{A} \pil \hhat{B}$ sends $X$ to the cone of
  $X$ over $E$, $Y = \cone_E(X)$, consisting of all $C \sus B$ such that
  $C \cap A \in X$.
\item[c.] $f^{\ii \ii} : \hhat{A} \pil \hhat{B}$ sends
  $X$ to the union of cones $\cone_E(X) \cup \cone_A(\pdel(E))$ where
  $\pdel(E)$ is the boundary of the simplex on $E$.
\end{itemize}

\medskip
The Stanley-Reisner ideals in $k[x_B]$ of
$f^{\uu \uu}(X), f^{**}(X), f^{\ii\ii}(X)$ are the ideals:
\[ (I_X) + (x^\mult_E), \quad (I_X), \quad (I_X) \cdot (x^\mult_E). \]
\end{proposition}

\begin{proof}
  Statements a, b, and c on simplicial complexes
  are consequences of Theorem \ref{thm:AB-fpp}.

  As for the statement on the ideal of $Y = f^{\ii \ii}(X)$,
  the ideal of $\cone_E(X)$ is $(I_X)$, the ideal of
  $\cone_A(\pdel(E))$ is $(x^\mult_E)$. Thus
  the ideal of the union is the intersection of these two ideals.
  \end{proof}

  In the simplest case we have the following.
  \begin{corollary} Suppose $E = B \backslash A = \{e \}$ consists
    of one element. 
    Letting $\Delta(A)$ be the full simplex on $A$, we have:
    \[ f^{\uu \uu}(X) = X, \quad f^{**}(X) = \cone_{\{e\}}(X),
      \quad f^{\ii \ii}(X) = \cone_{\{e\}}(X) \cup \Delta(A). \]
  \end{corollary}
  
  \begin{remark} Given $X$ in $\hhat{A}$.  By Lemma \ref{lem:adj-fiber},
   the elements $Y$ of $\hhat{B}$ such that $f^{*\uu}(Y) = X$ are
   those in the range
 \[ f^{\uu \uu}(X) = X \text{ (considered on $B$)} \sus Y \sus 
    f^{**}(X) = \cone_E(X). \]

   Similarly the elements $Y$ of $\hhat{B}$ such that $f^{*\ii}(Y) = X$ are
   those in the range
   \[ f^{* *}(X) = \cone_E(X) \sus Y \sus
     f^{\ii \ii}(X) = \cone_E(X) \cup \cone_A(\pdel(E)). \]
  \end{remark}

\section{Surjections}
\label{sec:surj}

 Let $f : A \pil B$ be a surjection. We let $A_b$ be the inverse image
  of $b \in B$. Then $A = \coprod_b A_b$ is a partition of $A$.
  In this section $X$ is a simplicial complex on $A$ and $Y$ a simplicial
 complex on $B$.

\subsection{The functors $f^{\uu *}$ and $f^{\ii *}$}

For $D \sus A$ recall the notion $\core_f(D) \sus B$.
Note that given
\[ C = \{b_1, \ldots, b_t \} \sus B = \{b_1, b_2, \ldots, b_m\}, \]
the maximal $D$ such that $\core_f (D) = C$
are precisely the facets of the join
\[ \Delta(\cup_{1}^t  A_{b_i}) * \pdel(A_{b_{t+1}}) * \cdots * \pdel(A_{b_m}). \]

\begin{definition} Given a surjection $f : A \pil B$. 
  A simplicial complex $X$ on $A$ is a {\it lower $f$-complex}
  if for every $b \in B$ and facet $F$ of $X$ then $F \cap A_b$
  is either $A_b$ or empty.
  Equivalently every facet $F$ of $X$ is $f^{-1}(\core_f(F))$.

  The simplicial complex $X$ is an {\it upper $f$-complex}
  if whenever $D$ is a face of $X$ with $C = \core_f(D)$, then
  every $D^\prime \sus A$ with $\core_f(D^\prime) = C$ is a
  face of $X$. Thus again $X$ is determined by its cores. 
\end{definition}

We get one-to-one correspondences:
\begin{align*}
\text{upper $f$-complexes on $A$ } & \overset{1-1}{\leftrightarrow}
  \text{ simplicial complexes on $B$} \\
  & \overset{1-1}{\leftrightarrow}
  \text{ lower $f$-complexes on $A$} 
\end{align*}

For $Y$ on $B$ denote by $Y_f$ the corresponding lower $f$-complex on $A$,
and by $Y^f$ the corresponding upper $f$-complex.
Note that the facets of $Y$ and $Y_f$ are in bijection 
by $S \mapsto f^{-1}(S)$, and the co-facets of
$Y$ and $Y^f$, are in bijection by $N \mapsto f^{-1}(N)$.

\begin{proposition} \label{pro:surj-fp*}
  Given a surjection $f : A \pil B$. 
  \begin{itemize}
  \item[a.]
    The map $f^{\uu *} : \hhat{B} \pil \hhat{A}$ sends $Y$ in $\hhat{B}$
    to the lower $f$-complex $X = Y_f$.
   \item[b.] The map $f^{\ii *} : \hhat{B} \pil \hhat{A}$ sends $Y$
     in $\hhat{B}$ to the upper $f$-complex $X = Y^f$.
   \end{itemize}
   \end{proposition}

   \begin{corollary} \hskip 1mm
     \begin{itemize}
\item[a.] The ideal of $Y_f$ in $k[x_A]$ is the ideal generated by
all product ideals $\prod_{b \in C}(x^{\pluss}_{A_b}) $ for $y^{\mult}_C$ in
$I_Y \sus k[y_B]$. 
\item[b.] 
   The ideal of $Y^f$ in $k[x_A]$ is the ideal generated by
   all monomials $x^\mult_{f^{-1}(C)} = \prod_{b \in C} x^\mult_{A_b}$ for all
   $y^\mult_C$ in
   $I_Y \sus k[y_B]$.
 \end{itemize}
 \end{corollary}

\begin{proof}
The statement on simplicial complexes follows by Theorem \ref{thm:AB-fp*}. 

To show the statement for ideals, let $Y$ correspond to the cut
$(\cC, \cV)$ of $\hhat{B}$. 
That $D \sus A$ is not in $Y^f$ means that
    $f^{-1}(C) \sus D$ for some $C \in \cV$. But this means that
    $x^\mult_D$ is in the ideal of $Y^f$ iff it is divisible by $\prod_{b \in C}
    x^\mult_{A_b}$ for some $y^\mult_C \in I_Y$. 
  \end{proof}

  \begin{remark} For a Stanley-Reisner ideal $I_Y$, taking a variable
    $x_b$ and replacing each occurrence in the minimal generators
    with a product of variables $x^\mult_{A_b}$ is a natural operation. But
    in textbooks or in the  literature it is hard to find explicit statements on
    how this changes the simplicial complex. The above explicitly gives
    this as $Y^f$ (although this is of course easy to work out by anyone
    needing it).
  \end{remark}

  \begin{example} \label{eks:surjU-join} Write
    $B = \{b_1, \ldots, b_m,\}$.
    Then $f^{*\uu}(\{\emptyset\}) = \{\emptyset \}$ and
    $f^{*\ii}(\{\emptyset\})$ is the join
    \begin{equation} \label{eq:surj-sisphere}
      X = \pdel(A_{b_1}) * \pdel(A_{b_2}) * \cdots * \pdel(A_{b_m}).
      \end{equation}
    (Here $Y = \{\emptyset\}$,
    $I_Y = (x_{b_1}, x_{b_2}, \ldots, x_{b_m})$ and
    $I_X = (x^\mult_{A_{b_1}}, \ldots, x^\mult_{A_{b_m}})$,
    confer the remark above.) 
    The join \eqref{eq:surj-sisphere}
    is a join of simplicial spheres, and is thus itself a
    simplicial sphere. The article \cite{Pol} shows its significance
    (and the present article its naturality):
    The full-dimensional  Cohen-Macaulay subcomplexes
    of this  join are shown to be constructible simplicial balls.
    Furthermore they are precisely the simplicial complexes whose
    Stanley-Reisner ideals are polarizations (in the most general sense,
    \cite{AFL})
    of Artin monomial ideals ${\mathfrak j} \sus k[y_B]$ such
    that ${\mathfrak j} \supseteq (y_{b_1}^{|A_{b_1}|}, \ldots, y_{b_m}^{|A_{b_m}|})$.
  \end{example}

  \begin{remark}
    The minimal free resolutions of $k[x_A]/I_{Y^f}$ is obtained from
    the minimal free resolution of $k[y_B]/I_Y$ by replacing
    each occurrence of a variable $x_b$ with a product $x^{A_b}$.
    In particular $I_Y$ and $I_{Y^f}$ have the same projective dimension.
  \end{remark}

  \begin{remark}
    The ideals $I_Y$ and $I_{Y_f}$ have the same regularity.
    If $Y = Z^\vee$ is the Alexander dual of $Z$, by Proposition \ref{pro:ad-ad}
    \[ (Z^\vee)_f = f^{*\uu}(Z^\vee) = f^{*\ii}(Z)^\vee = (Z^f)^\vee. \]
    Since Alexander duality transfers projective dimension to regularity,
    $I_{Z^\vee}$ and $I_{(Z^\vee)_f}$ have the same regularity.

    If $f : A \pil B$ is surjective with $f^{-1}(b) = \{a_1, \ldots, a_m\}$,
    and all other fibers of cardinality $1$, the minimal free resolution
    of $I_{Y_f}$ is obtained from that of $I_Y$, by replacing
    each term $S(-y_by_R)$ in the free resolution, by a Koszul complex
    \[ \oplus_{1 \leq i \leq m} S(-x_{a_i}) \vpil \oplus_{1 \leq i < j \leq m}
      S(-x_{a_i}x_{a_j}) \vpil \cdots ,  \]
    multiplied with $x_R$. We thank L.Br{\aa}then Knudsen for this
    observation.
    \end{remark}
  
\subsection{The functors $f^{\uu \uu}, f^{**}, f^{\ii \ii}$}

 \begin{proposition} Given a surjection $f : A \pil B$. 
    \begin{itemize}
    \item[a.] $f^{\uu \uu} : \hhat{A} \pil \hhat{B}$ sends $X$ on $A$
      to $Y$ on $B$ consisting of all $f(D)$ where $D \in X$. This is
      the smallest $Y$ in $\hhat{B}$ such that $Y_f \supseteq X$.
     
    \item[b.] $f^{**} : \hhat{A} \pil \hhat{B}$ sends $X$ to $Y$ consisting
      of $C \sus B$ such that $f^{-1}(C) \in X$. This is the
      largest $Y$ on $B$ such that $Y_f \sus X$.
      It is also the smallest $Y$ such that $Y^f \supseteq X$. 

    \item[c.] $f^{\ii \ii} : \hhat{A} \pil \hhat{B}$ sends $X$
      to $Y$ which is the largest $Y$ on $B$ such that
      $Y^f \sus X$. 
      \end{itemize}
    \end{proposition}

    \begin{proof} 
      The statement follows by Theorem \ref{thm:AB-fpp} and Lemma
      \ref{lem:poset-maxmin}.
\end{proof}

If $x^\mult_D \in k[x_A]$ is a squarefree monomial, let its {\it $f$-core} be
the monomial $x^\mult_C \in k[y_B]$ where $C = \core_f(D)$.

\begin{corollary} Given a simplicial complex $Y$ on $B$.
  \begin{itemize}
\item[a.] For $Y = f^{\uu \uu}(X)$, the ideal $I_Y \sus k[y_B]$ of $Y$
  is generated by the monomials $y^\mult_C$ (for $C \sus B$) such that
  $\prod_{b \in C} (x^\pluss_{A_b})$ is in $I_X \sus k[x_A]$.

\item[b.] For $Y = f^{* *}(X)$, its ideal $I_Y \sus k[y_B]$ of $Y$ is
  generated by monomials $y^\mult_C$ such that the monomial
  $x^\mult_{f^{-1}(C)} = \prod_{b \in C} x^\mult_{A_b}$ is in $I_X \sus k[x_A]$.
\item[c.] For $Y = f^{\ii \ii}(X)$, its ideal $I_Y \sus k[y_B]$
is generated by all {\em $f$-cores} of monomials in $I_X \sus k[x_A]$.
\end{itemize}

\end{corollary}

\begin{proof}
a.  That $C \sus B$ is {\it not} in $Y$ means that there is no
$D \sus f^{-1}(C)$ mapping bijectively to $C$ with $D \in X$.
This means that $y^\mult_C \in I_Y$ iff every $x^\mult_D$ dividing
$x^\mult_{f^{-1}(C)}$ and
with $D$ mapping bijectively to $C$, is in $I_X$. But such $x^\mult_D$
are precisely the monomials in the ideal $\prod_{b \in C} (x^+_{A_b})$.

b.  That $C \sus B$ is {\it not} in $Y$ means that $f^{-1}(C)$ is {\it not}
in $X$. Hence $y^\mult_C \in I_Y$ iff $x^\mult_{f^{-1}(C)}$ is in $I_X$.

c. This is by Lemma \ref{lem:AB-Y}b.
\end{proof}


By Lemma \ref{lem:adj-fiber} we get the following:

\begin{corollary} Given a simplicial complex $Y$ on $B$. 
  \begin{itemize}
  \item[i.] The simplicial complexes $X$ on $A$
    such that $f^{**}(X) = Y$ are precisely the $X$ such
  that $Y_f \sus X \sus Y^f$.    
\item[ii.] $X = Y_f$ is the unique maximal simplicial complex on $A$ such
  that $f^{\uu \uu}(X) = Y$.
\item[iii.] $X = Y^f$ is the unique minimal simplicial complex on $A$ such
  that $f^{\ii \ii}(X) = Y$.
  \end{itemize}
\end{corollary}

\subsection{Simplicial complexes from sections}

For a simplicial complex $X$ on $A$, its {\it support} is the
set of elements $a$ of $A$ which are vertices of $X$, i.e $\{a\}$ is
face of $X$. 
The {\it co-support} of $X$ is the set of vertices $a$ of $A$ such
that $A \backslash \{a\}$ is a face of $X$.

\medskip For a simplicial complex $Y$ on $B$, we have seen
how  $f^{*p}(Y)$, with $p = \uu, \iig$
represents extremal elements for the functors $f^{\uu \uu}, f^{**}, f^{\ii \ii}$. 
In particular $X = f^{*\uu}(Y) = Y_f$ is a unique
maximal element with $f^{\uu \uu}(X) = Y$.
However there is no unique minimal element mapping  to $Y$ by $f^{\uu \uu }$.

The following however describes a class of such.
We take a section $s : B \pil A$ of the surjection $f : A \pil B$,
so $f \circ s = \id_B$. Given a simplicial complex $Y$ on $B$,
let $Y_{s(B)}$ be its transferal to $s(B)$ via the bijection
$s : B \pil s(B)$.

\begin{proposition} Consider sections $s : B \pil A$ of the surjection
  $f : A \pil B$.
\begin{itemize}
\item[a.] The $X = s^{\uu \uu}(Y) = Y_{s(B)}$ are the $X$ on $A$
  with minimal support such that $f^{\uu \uu}(X) = Y$.
  The faces of $X$ and $Y$ are in bijection.
\item[b.] For $X = s^{**}(Y) = \cone_{A \backslash s(B)}(Y_{s(B)})$ we
  have $Y_f \sus X \sus Y^f$. For these  $X$ the facets of $X$ and $Y$
  are in bijection, and similarly the co-facets are in bijection. 
  \item[c.]  The 
$X = s^{\ii \ii}(Y) = \cone_{A \backslash s(B)}( Y_{s(B)}) \cup
  \cone_{s(B)}(\pdel(A \backslash s(B)))$ 
  are the $X$ on $A$ with maximal co-support
  such that $f^{\ii \ii}(X) = Y$.
  The non-faces of $X$ and $Y$ are in bijection.
  \end{itemize}
\end{proposition}

\begin{proof}
  For part a., it is clear that $f^{\uu \uu}(X) = Y$, due to functoriality
  and $f \circ s = \id_B$. These $X$ also have the minimal possible
  support cardinality of such $X$, which is $|B|$. If an $X$ mapping by
  $f^{\uu \uu}$ to $Y$
  has such support cardinality the support must be $s(B)$ for a section,
  and it is clear that $X = s^{\uu \uu}(Y)$.

  Part b. is clear. 
  The statement of part c. follows by a. and applying Alexander duality
  Proposition \ref{pro:ad-ad}.
\end{proof}

\medskip
\begin{remark}
  Relating to observation b. above, often there will be many more
  $Y_f \sus X \sus Y^f$ such that the facets of $X$ and $Y$ are in bijection
  and the co-facets of $X$ and $Y$ are in bijection.

  One large such class of $X$ are those that are successive {\it separations}
  of $Y$, \cite{AFL}. In this class for each co-facet $N$ of $X$
  we have a bijection $N \pil f(N)$ to the corresponding co-facet of $Y$.
  However there are also $X$ giving facet and co-facet bijections
  such that corresponding co-facets may have distinct cardinalities,
  see the following example.
\end{remark}

\begin{example}
  Let $B = \{a,b,x,y,r\}$ and $A = \{a,b,x,y,r_1, r_2\}$ with the
  natural map $A \pil B$.

  Let $Y$ be the simplicial complex with:
  \begin{itemize}
    \item Facets:
  $\{a,x,y\}, \, \{b,x,y\}, \, \{r,x\}, \, \{r,y\}$.
  \item Its co-facets are then:
  $\{a,r\}, \, \{b,r\}, \, \{a,b\}, \, \{r,x,y\}$.
\end{itemize}
 
Let $X$ be the simplicial complex with:
\begin{itemize}
\item Facets: $\{a,r_1, x,y\}, \, \{b,r_1,x,y\}, \, \{r_1,r_2,x\},
  \, \{r_1, r_2, y\}$.
  \item Its co-facets are then:
  $ \{a,r_2\}, \, \{b,r_1\}, \, \{a,b\}, \, \{r_1,r_2,x,y\}$.
\end{itemize}

  Then $Y = \core_f X$, and both facets and co-facets are in bijection.
\end{example}

\section{Products of simplicial complexes}
\label{sec:prod}

Given a simplicial complex $X$ on $A$ and $Y$ on $B$, we discuss
various products of these simplicial complexes on respectively
the union $A \cup B$ and the product $A \times B$. 

\subsection{Products on the union $A \cup B$}
For disjoint sets $A$ and $B$, the union $A \cup B$ is the
categorical coproduct, and comes with inclusions
\[ i_A : A \pil A \cup B, \quad i_B : B \pil A \cup B. \]
This gives maps
\[i_A^{\uu \uu}, i_A^{* *}, i_A^{\ii \ii} : \hhA \lpil (A \cup B)^{\hhat{}},
\quad i_B^{\uu \uu}, i_B^{* *}, i_B^{\ii \ii} : \hhB \lpil (A \cup B)^{\hhat{}}.
\]
For simplicial complexes $X$ on $A$ and $Y$ on $B$, we may apply
these maps and get simplicial complexes on $A \cup B$. Taking
the join or meet of these simplicial complexes we 
get various products of $X$ and $Y$, which will be simplicial complexes
on $A \cup B$.

\begin{theorem} Let $X$ be a simplicial complex on $A$, and $Y$ on $B$. 
\begin{itemize}
\item[a.] The join of $i_A^{\uu \uu}(X)$ and $i_B^{\uu \uu}(Y)$ is
the disjoint union of the two simplicial complexes. It consists
of $C \sus A \cup B$ such that $C \in X$ \underline{or} $C \in Y$.
(Their meet is $\{\emptyset \}$ if $X,Y \neq \emptyset$.) 
\item[b.] The meet of $i_A^{* *}(X)$ and $i_B^{* *}(Y)$ is
the external join $X * Y$. It  consists of $C \sus A \cup B$ such that
$C \cap A \in X$ {\underline{and}} $C \cap B  \in Y$.
\item[c.] The join of $i_A^{* *}(X)$ and $i_B^{* *}(Y)$
  consists of $C \sus A \cup B$ such that
$C \cap A \in X$ {\underline{or}}  $C \cap B  \in Y$. 
\item[d.] The meet of  $i_A^{\ii \ii}(X)$ and $i_B^{\ii \ii}(Y)$ consists
  of $C \sus A \cup B$ such that either:
  \begin{itemize}
  \item[i.] $C \cap A \subsetneq A$ and $C \cap B \subsetneq B$,
  \item[ii.] $C \cap B = B$ and $C \cap A \in X$,
  \item[iii.] $C \cap A = A$ and $C \cap B \in Y$.
  \end{itemize}
(Their join is the boundary $\pdel (A\cup B)$ when $X \neq \Delta(A)$ and
  $Y \neq \Delta(B)$.) 
\end{itemize}
\end{theorem}

\begin{proof}
  a. is clear since $i_A^{\uu \uu}(X)$ is just $X$ considered as
  a simplicial complex on $A \cup B$, with the vertices of $B$ as
  ``ghost vertices'',

\noindent b. and c. follow from 
  $i_A^{**}(X) = \cone_B(X)$ and $i_B^{**}(X) = \cone_A(Y)$,

\noindent d. follows from $i_A^{\ii \ii}(X)
  = \cone_B(X) \cup \cone_A(\pdel(B))$ and similarly for $i_B^{\ii \ii}(Y)$.
\end{proof}

 For the Stanley-Reisner ideals in $k[x_{A \cup B}]$ we have the following:

 \begin{corollary} Following the cases above, the Stanley-Reisner ideals
   are as follows:
\begin{itemize}
\item[a.] $(I_X) + (I_Y) + (x^+_A) \cdot (y^+_B) =
  (I_X) + (I_Y) + (x_a y_b)_{a \in A, b \in B}.$
\item[b.] $(I_X) + (I_Y)$,
\item[c.] $(I_X) \cap (I_Y) = (I_XI_Y)$,
\item[d.] $(I_X \cdot y^\mult_B) + (x^\mult_A \cdot I_Y)$.
\end{itemize}
\end{corollary}

\begin{proof} We do the last case d., as the others are similar.
  If $D \cup C \sus A \cup B$ (with $D \sus A$ and $C \sus B$)
  is not in the simplicial complex, then
  either $D= A$ and $C \not \in Y$, which means
  that $x^\mult_A y^\mult_C \in x^\mult_A \cdot I_Y$, or 
$C = B$ and $D \not \in X$, which means
that $x^\mult_D y^\mult_B  \in I_X \cdot y^\mult_B$.
\end{proof}

\subsection{Products on the cartesian product $A \times B$}
The cartesian product $A \times B$ is the categorical product in
the category of sets. The projection maps $p_A : A \times B \pil A$
and $p_B : A \times B \pil B$ induce maps:
\[ p_A^{\uu *}, p_A^{\ii *} : \hhA \pil (A \times B)^{\hhat{}}, \quad
  p_B^{\uu *}, p_B^{\ii *} : \hhB \pil (A \times B)^{\hhat{}}.
\]

For simplicial complexes $X$ on $A$ and $Y$ on $B$, we apply
these maps and get simplicial complexes on $A \times B$. Taking
the join or meet of these simplicial complexes we
get various products of $X$ and $Y$, which will be simplicial complexes
on $A \times B$. (For some of the terminology below, see Subsection
\ref{subsec:adj-term}.)

\begin{theorem}  Let $X$ be a simplicial complex on $A$, and $Y$ on $B$. 
  \begin{itemize}
  \item[a.] The meet of $p_A^{\uu *}(X)$ and $p_B^{\uu *}(Y)$ is the
    simplicial complex on $A \times B$ generated by the faces
    $D \times C$ where $D \in X$ {\underline{and}} $C \in Y$.
  \item[b.] The join of  $p_A^{\uu *}(X)$ and $p_B^{\uu *}(Y)$ is the
    simplicial complex on $A \times B$ generated by the faces
    $D \times C$ where $D \in X$ {\underline{or}} $C \in Y$.
   \item[c.] The meet of  $p_A^{\ii *}(X)$ and $p_B^{\ii *}(Y)$ is the
     simplicial complex on $A \times B$ whose  non-face complements
     is generated by $U \times V$ where $U$ is a non-face complement
     of $X$ {\underline{or}} $V$ is a non-face complement of $Y$.
     (By ``generated'' we mean taking subsets.)
 \item[d.] The join of  $p_A^{\ii *}(X)$ and $p_B^{\ii *}(Y)$ is the
     simplicial complex on $A \times B$ whose non-face complements
     is generated by $U \times V$ where $U$ is a non-face complement
     of $X$ {\underline{and}} $V$ is a non-face complement of $Y$. 
   \end{itemize}
 \end{theorem}

 \begin{proof}
a. and b. $p_A^{\uu *}(X)$ is generated by faces $C \times B$, such
   that $C \in X$, and similarly for $p_B^{\uu *}(Y)$.

   c. and d. $p_A^{\ii *}(X)$ has non-face complements generated
   by $U \times B$, where $U$ is a non-face complement of $X$,
   and similarly for $p_B^{\uu *}(Y)$.
 \end{proof}

 For the Stanley-Reisner ideals in $k[z_{A \times B}]$ we have the following:

\begin{corollary} We follow the cases above:
  \begin{itemize}
  \item[a.] The ideal is generated by all $z^\mult_S$ where $S \sus A \times B$
    with $x^\mult_{p_A(S)} \in I_X$ or $y^\mult_{p_B(S)} \in I_Y$,
  \item[b.] The ideal is generated by all $z^\mult_S$ where $S \sus A \times B$
    with $x^\mult_{p_A(S)} \in I_X$ and $y^\mult_{p_B(S)} \in I_Y$,
 \item[c.] The ideal is generated by all $z^\mult_S$ where $S \sus A \times B$
    with $x^\mult_{p_A(S^c)^c} \in I_X$ or $y^\mult_{p_B(S^c)^c} \in I_Y$,
  \item[d.] The ideal is generated by all $z^\mult_S$ where $S \sus A \times B$
    with $x^\mult_{p_A(S^c)^c} \in I_X$ and $y^\mult_{p_B(S^c)^c} \in I_Y$,
  \end{itemize}
\end{corollary}

\begin{proof}
 We do case d. The others are similar or simpler.  
 The ideal is generated by $z^\mult_{\tilde D}$ where
 $\tilde{D}$ is a non-face of the join.
In the cut $(\tilde{D}, \tilde{U})$,
the non-face complement $\tilde{U}$ is then in $U \times V$
where $U$ is a non-face complement  of $X$ and $V$ a
non-face complement of $Y$. This means that $x^\mult_{U^c}$ is in $I_X$ and
$y^\mult_{V^c}$ is in $I_Y$. Setting $S = \tilde{D}$ gives the statement. 
\end{proof}

\bibliographystyle{amsplain}
\bibliography{biblio}
\end{document}